\documentclass[12pt,a4paper]{article}
\usepackage{theorem}
\usepackage{graphicx}
\newtheorem{lemma}{Lemma}%[chapter]

\newtheorem{theorem}[lemma]{Theorem}
\newtheorem{corollary}[lemma]{Corollary}
{\theorembodyfont{\upshape}\newtheorem{probl}[lemma]{Problem}
{\theorembodyfont{\upshape}}
{\theorembodyfont{\upshape}}
{\theorembodyfont{\upshape}\newtheorem{notee}[lemma]{Note}}
{\theorembodyfont{\upshape}\newtheorem{exampl}[lemma]{Example}}
{\theorembodyfont{\upshape}}

\newenvironment{example}{\begin{exampl}}{\hfill$\Box$\end{exampl}}

\newenvironment{note}{\begin{notee}}{\hfill$\Box$\end{notee}}
\newcommand{\C}{{\bf C}}

\newcommand{\N}{{\bf N}}

\newcommand{\R}{{\bf R}}

\newcommand{\rme}{{\rm e}}
\newcommand{\rmd}{{\rm d}}

\newcommand{\cD}{{\cal D}}

\newcommand{\cL}{{\cal L}}

\newcommand{\cW}{{\cal W}}
\newcommand{\wt}[1]{\widetilde{#1}}

\newcommand{\tH}{\wt{H}}
\newcommand{\tV}{\wt{V}}
\newcommand{\sig}{\sigma}
\newcommand{\alp}{\alpha}
\newcommand{\bet}{\beta}
\newcommand{\gam}{\gamma}

\newcommand{\lam}{\lambda}

\newcommand{\eps}{\varepsilon}

\newcommand{\zet}{\zeta}

\newcommand{\grad}{\nabla}

\newcommand{\Dom}{{\rm Dom}}

\newcommand{\Spec}{{\rm Spec}}

\newcommand{\Ker}{{\rm Ker}}
\newcommand{\Ran}{{\rm Ran}}

\newcommand{\norm}{\Vert}

\newcommand{\supp}{{\rm supp}}

\newcommand{\lin}{{\rm lin}}

\newcommand{\Schrodinger}{Schr\"odinger }

\newcommand{\la}{{\langle}}
\newcommand{\ra}{{\rangle}}
\newcommand{\pr}{\prime}

\newenvironment{eq}{\begin{equation}}{\end{equation}}

\newcommand{\eqref}[1]{(\ref{#1})}
\newenvironment{proof}{\textbf{Proof}}{\hfill$\Box$}

\newenvironment{choices}{ \left\{ \begin{array}{ll} }{\end{array}\right.}
\newcommand{\move}[1]{}
\hyphenation{hypo-thesis hypo-theses}
\usepackage{amssymb}
\usepackage{hyperref}
\title{Singular \Schrodinger operators\\
 in one dimension\\
arXiv version}
\author{E. B. Davies}
\date{16 December 2011}
\begin{document}
\maketitle
%%%%%%%%%%%%%%%%%%%%%
\begin{abstract}
We consider a class of singular \Schrodinger operators $H$ that act in $L^2(0,\infty)$, each of which is constructed from a positive function $\phi$ on $(0,\infty)$. Our analysis is direct and elementary. In particular it does not mention the potential directly or make any assumptions about the magnitudes of the first derivatives or the existence of second derivatives of $\phi$. For a large class of $H$ that have discrete spectrum, we prove that the eigenvalue asymptotics of $H$ does not depend on rapid oscillations of $\phi$ or of the potential. Similar comments apply to our treatment of the existence and completeness of the wave operators.
\end{abstract}



\section{Introduction}\label{intro}

The study of \Schrodinger operators traditionally starts by specifying the class of potentials to be studied. In this paper we start instead from a positive function $\phi$ (considered as a fundamental solution of a so far undefined \Schrodinger operator) and reconstruct the theory without making any assumptions about the existence of a second derivative of $\phi$. The success of this approach allows one to deal with a large class of \Schrodinger operators with distributional potentials, which is similar to the classes studied by Savchuk and Shkalikov in \cite{SS2}. It is also relevant to \Schrodinger operators with rapidly oscillating potentials, such as $V(x)=x^\alp \sin(x^\bet)$, where $\alp\in\R$ and $\bet>1$. The literature on this subject is extensive; \cite{AEZ,B-AD, comb,ET,HM,matv,ming,nest,pear,SS1,SS2,skrig,white} is a selection of a few of the early papers and other, much more recent, ones that contain many earlier references.

Starting from $\phi$ rather than from the potential $V$ makes the analysis of the \Schrodinger operator $H$ much simpler. This is justifiable if one takes for granted earlier work that details how to deduce asymptotic properties of $\phi$ from assumed properties of $V$; we summarize some of the important results on this subject, although the body of the paper does not depend on them. If $V\in L^1(0,\infty)$ then for all large enough $c>0$ the \Schrodinger equation $-\phi^{\pr\pr} +(V+c^2)\phi=0$ has a unique positive solution $\phi\in C^1[0,\infty)$ such that $\phi(x)\sim \rme^{-cx}$ and $\phi^\pr(x)\sim -c\rme^{-cx}$ as $x\to\infty$; this solution is obtained by solving the Jost equation
\[
f(x)=\rme^{-cx}+\frac{1}{c}\int_x^\infty \sinh(c(y-x)) V(y)f(y)\, \rmd y.
\]
See, for example, \cite{atk} and \cite[Theorem~XI.57]{RS3}. If $V$ is highly oscillatory but not in $L^1(0,\infty)$ then in certain special situations it is still possible to solve the \Schrodinger equation and develop a detailed spectral or scattering theory, depending on the assumptions about $V$ or $\phi$; see \cite[Appx.~2,3 to XI.8]{RS3} and \cite{comb,matv,pear,skrig,white}.

We carry out the analysis for an operator defined on $L^2(0,\infty)$, but the methods can be applied to finite intervals with obvious modifications. Our constructions assume that one is given a function $\phi\in L^2(0,\infty)$ which is positive and continuous on $[0,\infty)$ with $\phi^\pr\in L^2_{{\rm loc}}[0,\infty)$. Although further bounds on $\phi$ are imposed later, we never require any local or global bounds on the size of $\phi^\pr$, nor the existence of $\phi^{\pr\pr}$.

Theorems~\ref{Htheorem} and \ref{Hgamtheorem} concern a second order differential operator $H$ acting in $L^2(0,\infty)$ according to the formula
\[
Hf=\phi^{-1}(\phi^\pr f-\phi f^\pr)^\pr
\]
and subject to Dirichlet and then Robin boundary conditions at $0$; the precise domain of $H$ is determined for each choice of boundary condition. Corollary~\ref{HEScoro} proves that if $\phi$ converges super-exponentially to $0$ as $x\to\infty$ in a certain precise sense then the essential spectrum of $H$ is empty. We also determine the order of growth of the eigenvalues of $H$.  Theorem~\ref{scatt-theorem} provides conditions for the existence and completeness of the wave operators. Once again, we do not assume any bounds on $\phi^\pr$, or on the potential if one exists.

We will not use the following alternative expressions for $H$, but they serve to connect our approach to the `regularization method' described in \cite[sects.~1,~2]{HM}, whose notation we follow; see also \cite{AEZ,BE,EGNT,SS1,SS2} and \cite[Sect.~9.1]{tes}. One may rewrite $H$ in the form
\begin{eqnarray}
Hf&=&-(f^\pr-\tau f)^\pr-\tau(f^\pr-\tau f)\\
&=& -\left(\frac{\rmd}{\rmd x}+\tau\right)\left(\frac{\rmd}{\rmd x}-\tau\right)
\end{eqnarray}
where $\tau=\phi^\pr/\phi$. This factorization, known to Jacobi, is equivalent to the factorization (\ref{G=MM}) of the resolvent $G$ which underlies much of our analysis in Sections~\ref{sectSA} and \ref{sectST}.

If one puts
\[
\sig(x)=\tau(x)+\int_0^x \tau(y)^2 \, \rmd y
\]
for all $x\geq 0$, then one obtains
\[
Hf= -(f^\pr -\sig f)^\pr-\sig f^\pr.
\]
If there is a potential $V$ then $\phi^{\pr\pr}/\phi=V$, $\sig$ is the indefinite integral of $V$ and $\tau$ is a solution of the Riccati equation $\tau^\pr+\tau^2=V$. The relationship between solutions of Riccati differential equations and of linear second order differential equations has been known for centuries; we refer to \cite{BL,reid} and \cite[Chapter~XIII.3B]{RS4} for some aspects of the theory of the Riccati equation and its applications. The condition $\sig\in L^2_{{\rm loc}}[0,\infty)$ corresponds to $V\in W^{-1,2}_{{\rm loc}}[0,\infty)$, while $\sig$ has bounded variation on every interval $[0,a]$ if and only if $V$ is a countably additive signed Borel measure.

The next section is taken up with certain preliminaries that are needed later. In Section~\ref{sectDBC} we consider the Green functions and use them to define the related second order differential operators when Dirichlet boundary conditions are imposed at $0$. Section~\ref{sectRBC} studies the same problem subject to Robin boundary conditions with a complex parameter.

\section{Preliminaries}

We write down a number of identities involving functions defined on $[0,\infty)$ below. If we do not refer to the variable explicitly, we mean that the two functions concerned are equal almost everywhere, which always means with respect to Lebesgue measure. If the functions involved are continuous this implies that they are equal everywhere. When we refer to something being true locally in $[0,\infty)$, we mean that it holds for all intervals of the form $[0,a]$. We need the following standard lemmas relating to weak derivatives in one dimension. Following \cite[Lemma~7.1.1]{STDO}, we say that $f\in \cW=W^{1,2}_{{\rm loc}}[0,\infty)$ if $f$ is continuous on $[0,\infty)$ and there exists a function $g\in L^2_{{\rm loc}}[0,\infty)$ such that
\begin{eq}
f(x)=f(0)+\int_0^x g(s)\, \rmd s \label{weakderdef}
\end{eq}
for all $x\geq 0$. It is easy to prove that given $f$, only one such $g$ can exist, and if it does we call it the weak derivative of $f$ and write $f^\pr=g$.

\begin{lemma}
If $f\in\cW$ satisfies $f(x)>0$ for all $x\geq 0$ then
$f^\alp\in\cW$ for every $\alp\in \R$ and $(f^\alp)^\pr=\alp
f^{\alp-1}f^\pr$.
\end{lemma}

\begin{proof}
By the definition of $\cW$ it is sufficient to treat the case in which $x\in [0,a]$, for every $a>0$. The assumption on $f$ implies that it is continuous and hence implies the existence of $c>0$ such that $1/c\leq f(x)\leq c$ for all $x\in [0,a]$. There exists a sequence $f_n\in C^1[0,a]$ such that $f_n$ converges uniformly to $f$ and $f_n^\pr$ converges in the $L^2(0,a)$ norm to $f^\pr$ as $n\to\infty$. Therefore $1/(2c)\leq f_n(x)\leq 2c$  uniformly with respect to $x\in [0,a]$, for all large enough $n$. The fundamental theorem of calculus implies that
\[
f_n(x)^\alp=f_n(0)^\alp+\int_0^x \alp f_n(s)^{\alp-1}f_n^\pr(s)\,
\rmd s
\]
for all $n$ and all $x\in [0,a]$. Letting $n\to\infty$ yields
\[
f(x)^\alp=f(0)^\alp+\int_0^x \alp f(s)^{\alp-1}f^\pr(s)\, \rmd s,
\]
which completes the proof.
\end{proof}

\begin{lemma}\label{leibniz}
If $f,\, g\in\cW$ then $fg\in\cW$ and
\[
(fg)^\pr=f^\pr g+fg^\pr.
\]
\end{lemma}

\begin{proof}
We first note that the proof of the previous lemma applies to the case $\alp=2$ without assuming that $f>0$. The proof of the present lemma is completed by noting that
\[
fg=\frac{(f+g)^2}{4}-\frac{(f-g)^2}{4}.
\]
\end{proof}

\begin{lemma}
There exists a unique function $\psi\in \cW$ such that $\psi(0)=0$ and $\psi^\pr(x)\phi(x)- \phi^\pr(x)\psi(x)=1$ for almost all $x\geq 0$.
\end{lemma}

\begin{proof}
One may directly verify that the function
\begin{eq}
\psi(x)=\phi(x)\int_0^x\frac{\rmd s}{\phi(s)^2}\label{psidef}
\end{eq}
has the stated properties. If $\xi$ is another such function and $\sig=\psi-\xi$ then
$\sig^\pr\phi- \phi^\pr\sig=0$. Differentiating $\tau=\sig/\phi\in \cW$ one obtains $\tau^\pr=0$, so $\tau$ is constant. But $\tau(0)=(\psi(0)-\xi(0))/\phi(0)=0$. Therefore $\tau=0$ and $\xi=\psi$.
\end{proof}

Our next lemma justifies calling $\phi$ the subordinate function in $\cL=\lin\{\phi,\psi\}$ in the sense that it is the smallest function in $\cL$ at infinity, up to a multiplicative constant.
\begin{lemma}\label{psi/phi}
The function $\psi$ is positive on $(0,\infty)$. Moreover $\psi/\phi$ is strictly monotonic increasing on $(0,\infty)$ and
\[
\lim_{x\to \infty} \frac{\psi(x)}{\phi(x)}=+\infty.
\]
\end{lemma}

\begin{proof}
The first and second statement follow directly from (\ref{psidef}). The final statement follows from
\begin{eqnarray*}
x^2&=&\left( \int_0^x \rmd s\right)^2\\
&\leq& \int_0^x\phi(s)^2\, \rmd s\int_0^x\frac{\rmd s}{\phi(s)^2}\\
&\leq& \norm \phi\norm^2\int_0^x\frac{\rmd s}{\phi(s)^2}\\
&=& \norm \phi\norm^2 \frac{\psi(x)}{\phi(x)}.
\end{eqnarray*}
\end{proof}

If $f$ is a positive, twice differentiable function on some interval $I$ and $-f^{\pr\pr}+Vf=0$ then a simple calculation establishes that
\begin{eq}
V(x)=\frac{\rmd}{\rmd x} \left( \frac{f^\pr(x)}{f(x)}+\int_a^x%
\frac{f^\pr(s)^2}{f(s)^2}\, \rmd s\right)\label{Vdef}
\end{eq}
for any $a\in I$. This motivates the following lemma, in which we do not assume that any second derivatives exist. It also establishes the weak existence of a potential $V\in W_{{\rm loc}}^{-1,2}$, defined as the weak derivative of either side of (\ref{weakpot}).
\begin{lemma}
If $f,\, g$ are any two positive functions in $\cL={\rm lin} \{ \phi,\psi\}$ then there exists a constant $c$ such that
\begin{eq}
\frac{f^\pr(x)}{f(x)}+\int_a^x\frac{f^\pr(s)^2}{f(s)^2}\, \rmd s=%
\frac{g^\pr(x)}{g(x)}+\int_a^x\frac{g^\pr(s)^2}{g(s)^2}\, \rmd s+c\label{weakpot}
\end{eq}
for all $x\geq 0$.
\end{lemma}

\begin{proof}
It is sufficient to prove that
\[
h(x)=\frac{f^\pr(x)}{f(x)}-\frac{g^\pr(x)}{g(x)}%
+\int_a^x\left\{ \frac{f^\pr(s)^2}{f(s)^2}-\frac{g^\pr(s)^2}{g(s)^2}
\right\}\, \rmd s
\]
is differentiable with zero derivative, assuming  that $f,\, g$ are positive with $f^\pr g-g^\pr f=k$ for some constant $k$. After rewriting $h$ in the form
\[
h(x)=\frac{k}{f(x)g(x)}+\int_0^x k\frac{f^\pr(s)g(s)+g^\pr(s)f(s)}{f(s)^2g(s)^2}\, \rmd s,
\]
the result follows immediately.
\end{proof}

\begin{example}The following examples are of interest. In each case the associated potential may be calculated and for $\phi_4$ it is rapidly oscillating.
\begin{itemize}
\item $\phi_1(x)=\rme^{-cx}$ where $c>0$;
\item $\phi_2(x)=(1+x)^{-c}$ where $c>1/2$;
\item $\phi_3(x)=\exp( -(1+x)^c)$ where $c>0$;
\item $\phi_4(x)= \exp(-x-\sin(\rme^x))$.
\end{itemize}
The functions $\phi_1$ and $\phi_4$ obey the conditions in Theorem~\ref{Gexpbound} below. However, $\phi_2$ does not and the corresponding operator $G$ is not bounded; this follows from the fact that $\lim_{x\to\infty}V(x)=0$ in both cases, so $0$ lies in the essential spectrum of the differential operator $H=G^{-1}$. The function $\phi_3$ satisfies the conditions of Theorem~\ref{Gexpbound} if and only if $c\geq 1$, which is also the condition for $G$ to be bounded.
\end{example}

\section{Dirichlet boundary conditions}\label{sectDBC}

We use the following lemma in the proof of Lemma~\ref{L2Dlemma}.

\begin{lemma}\label{g-asympt-lemma}
Let $f\in L^2(0,\infty)$ and let
\[
g(x)=\int_0^\infty G(x,y)f(y)\, \rmd y
\]
where the Green function $G(\cdot,\cdot)$ is defined by
\begin{eq}
G(x,y)=\begin{choices}
\psi(x)\phi(y)&\mbox{if } 0\leq x\leq y <\infty,\\
\phi(x)\psi(y)&\mbox{if } 0\leq y\leq x <\infty.
\end{choices}
\label{Gdef}
\end{eq}
Then $g$ is continuous on $[0,\infty)$ and
\[
\lim_{x\to\infty} \frac{g(x)}{\psi(x)}=0.
\]
\end{lemma}

\begin{proof}\\
\textbf{Step~1} Define the operator $L$ by
\[
(Lf)(x)=\frac{1}{\psi(x)}\int_0^\infty G(x,y)f(y)\, \rmd y
\]
for all $x\geq 0$, where we assume that $f\in L^2(0,\infty)$. The formula
\[
(Lf)(x)=\int_x^\infty  \phi(y)f(x)\,\rmd y %
+\frac{\phi(x)}{\psi(x)}\int_0^x \psi(y)f(y)\,\rmd y
\]
implies $Lf\in C(0,\infty)$ by inspection. Using Lemma~\ref{psi/phi} we obtain
\begin{eqnarray*}
|(Lf)(x)|&\leq & \int_x^\infty  \phi(y)|f(x)|\,\rmd y %
+\int_0^x \frac{\phi(y)}{\psi(y)}\psi(y)|f(y)|\,\rmd y\\
&=&\int_0^\infty \phi(y)|f(y)|\, \rmd y\\
&\leq & \norm \phi\norm \, \norm f\norm.
\end{eqnarray*}
Therefore $L$ is a bounded operator from $L^2(0,\infty)$ to the space $C_b(0,\infty)$ of bounded continuous functions on $(0,\infty)$ with the uniform norm.\\
\textbf{Step~2} Let $L_c^2(0,\infty)$ denote the space of functions $f\in L^2(0,\infty)$ that have support in some interval $(a,b)$, where $0<a<b<\infty$. For such a function $x<a$ implies
\[
(Lf)(x)=\int_a^b \phi(y)f(y)\, \rmd y.
\]
Moreover $x>b$ implies
\[
(Lf)(x)=\frac{\phi(x)}{\psi(x)}\int_a^b \psi(y)f(y)\, \rmd y,
\]
and an application of Lemma~\ref{psi/phi} implies that $\lim_{x\to\infty} (Lf)(x)=0$. Therefore $f\in L^2_c(0,\infty)$ implies $Lf$ lies in the space $C_0[0,\infty)$ of continuous functions on $[0,\infty)$ that vanish at $\infty$.\\
\textbf{Step~3} Since $L^2_c(0,\infty)$ is norm dense in $L^2(0,\infty)$ and $C_0[0,\infty)$ is closed with respect to uniform limits in $C_b(0,\infty)$, we deduce by Step~1 that $Lf\in C_0[0,\infty)$ for all $f\in L^2(0,\infty)$. This yields the statement of the lemma.
\end{proof}

We will not use the hypothesis of the following theorem until Section~\ref{sectSA}, but include it as typical of theorems that ensure that $G(\cdot,\cdot)$ is associated with a bounded operator.

\begin{theorem}\label{Gexpbound}
Suppose that there exist positive constants $c,\, c_1,\, c_2$ such that
\[
c_1\rme^{-\sig(x)}\leq \phi(x)\leq c_2\rme^{-\sig(x)}
\]
for all $x\geq 0$, where $\sig$ is a $C^1$ function on $[0,\infty)$ such that $\sig^\pr(x)\geq c$ for all $x\geq 0$. Then the Green function $G(\cdot,\cdot)$ satisfies
\begin{eq}
0\leq G(x,y)\leq \frac{c_2^3}{2c c_1^3}\rme^{-c|x-y|}\label{Gbound}
\end{eq}
for all $x,\, y\geq 0$ and it is
is the integral kernel of a bounded operator on $L^2(0,\infty)$.
\end{theorem}

\begin{proof}
If $0\leq x\leq y$ there there exists $u\in (x,y)$ such that
\begin{eqnarray*}
\frac{\phi(y)}{\phi(x)}&\leq & \frac{c_2}{c_1}\rme^{-(\sig(y)-\sig(x))}\\
&=& \frac{c_2}{c_1}\rme^{-(y-x)\sig^\pr (u)}\\
&\leq & \frac{c_2}{c_1}\rme^{-c(y-x)}.
\end{eqnarray*}
Therefore
\begin{eqnarray*}
\psi(t)&= & \phi(t)\int_0^t \frac{\rmd s}{\phi(s)^2}\\
&=& \frac{1}{\phi(t)}\int_0^t \frac{\phi(t)^2}{\phi(s)^2}%
\, \rmd s\\
&\leq & \frac{1}{\phi(t)}\int_0^t\frac{c_2^2}{c_1^2}%
\rme^{-2c(t-s)}\, \rmd s\\
&\leq & \frac{c_2^2}{2c c_1^2\phi(t)}
\end{eqnarray*}
for all $t\geq 0$. If $0\leq x\leq y$ this implies that
\begin{eqnarray*}
0\, \leq \, G(x,y)&=& \psi(x)\phi(y)\\
 &=& \psi(x)\phi(x)\frac{\phi(y)}{\phi(x)}\\
 &\leq & \frac{c_2^2}{2c c_1^2}\,\frac{c_2}{c_1}\rme^{-c(y-x)}.
\end{eqnarray*}
After using the invariance of the kernel when exchanging $x$ and $y$ this yields (\ref{Gbound}). This bound implies that $G(\cdot,\cdot)$ is the kernel of a bounded operator on $L^p(0,\infty)$ for all $1\leq p\leq\infty$; see \cite[Cor.~2.2.19]{LOTS}. Moreover
\begin{eq}
\norm G\norm \leq \frac{c_2^3}{c^2 c_1^3}.\label{Gnormbound}
\end{eq}
\end{proof}

\begin{lemma}\label{Gboundedlemma}
If $G(\cdot,\cdot)$ is the kernel of a bounded operator, which we also call $G$, acting on $L^2(0,\infty)$ then $G$ is self-adjoint, $\Ker(G)=0$ and $\Ran(G)$ is dense. Moreover $G\geq 0$ in the spectral sense, that is $\Spec(G)\subseteq [0,\norm G\norm ]$.
\end{lemma}

\begin{proof}
The symmetry of the kernel implies that $G$ is self-adjoint. If $f\in L^2(0,\infty)$ then $Gf=0$ if and only if
\[
\psi(x)\int_x^\infty \phi(y)f(y)\, \rmd y%
+\phi(x)\int_0^x \psi(y)f(y)\, \rmd y=0
\]
for all $x\geq 0$. Differentiating this formula yields
\[
\psi^\pr(x)\int_x^\infty \phi(y)f(y)\, \rmd y%
+\phi^\pr(x)\int_0^x \psi(y)f(y)\, \rmd y=0
\]
for almost all $x\geq 0$. A simple linear combination of the last two formulae now yields
\[
\int_x^\infty \phi(y)f(y)\, \rmd y=0
\]
for almost all and then for all $x\geq 0$. Differentiating this identity leads to $f=0$. The density of the range of $G$ is equivalent to the vanishing of its kernel because $G$ is self-adjoint.

If $0\leq f\in L^2(0,\infty)$ then
\begin{eq}
\la Gf,f\ra=\la M f,M f\ra \label{G=MM}
\end{eq}
where $M$ is the operator with domain $L^2(0,\infty)$ defined by
\begin{eq}
(Mf)(x)=\int_x^\infty \frac{\phi(y)}{\phi(x)} f(y)%
\, \rmd y.\label{Mdef}
\end{eq}
This is proved by direct rearrangements of the order of integration in the triple integrals involved, this being justified by the fact that the integrands are all pointwise non-negative. If $f\in L^2(0,\infty)$ then $|(Mf)|\leq M|f|$ pointwise, so $\norm Mf\norm^2\leq \norm M|f|\norm^2=\la G|f|,|f|\ra \leq \norm G\norm\, \norm f\norm^2$. Therefore $Mf\in L^2(0,\infty)$ and $\norm M\norm \leq \norm G\norm^{1/2}$ and (\ref{G=MM}) holds for all $f\in L^2(0,\infty)$. The factorization of $G$ in (\ref{G=MM}) is equivalent to $G=LL^\ast$ where $L=M^\ast$ is given by
\[
(Lf)(x)=\int_0^x \frac{\phi(x)}{\phi(y)} f(y)\, \rmd y.
\]
We have proved that $\la Gf,f\ra\geq 0$ for all $f\in L^2(0,\infty)$, which is equivalent to the spectral positivity of $G$.
\end{proof}

\begin{note}
The function $D(x)=G(x,x)$ satisfies
\[
D(x)=\phi(x)^2\int_0^x \frac{\rmd s}{\phi(s)^2}
\]
for all $x\geq 0$. This implies that
\[
\frac{D^\pr(x)}{D(x)}-\frac{1}{D(x)}=2\, \frac{\phi^\pr(x)}{\phi(x)}.
\]
One could have started the analysis of this paper from the function $D$ rather than from $\phi$. See \cite[sect.~4]{DH} for further work in this direction.
\end{note}

For the rest of this paper we assume that $G(\cdot,\cdot)$ is the kernel of a bounded operator on $L^2(0,\infty)$, also denoted by $G$. Lemma~\ref{Gboundedlemma} implies that $G$ is the inverse of some other (possibly unbounded) self-adjoint operator, which we call $H$. It is well-known that if $\phi$ is twice differentiable and the associated potential $V=\phi^{\pr\pr}/\phi$ is sufficiently regular then
\[
(Hf)(x)=-f^{\pr\pr}(x)+V(x)f(x)
\]
on a suitable domain; moreover $H$ satisfies Dirichlet boundary conditions at $0$. We establish a similar result without assuming that $\phi$ is twice differentiable. The lemmas and theorems below refer to the following linear space of functions on $[0,\infty)$. We write $g\in\cD$ if the following conditions hold.
\begin{description}
\item[$\cD 1.$]\hspace{2em} $g\in  L^2(0,\infty)\cap \cW$;
\item[$\cD 2.$]\hspace{2em} $g(0)=0$;
\item[$\cD 3.$]\hspace{2em} $\phi g^\pr-\phi^\pr g\in \cW$;
\item[$\cD 4.$]\hspace{2em} $\phi^{-1}(\phi g^\pr-\phi^\pr g)^\pr\in L^2(0,\infty)$;
\item[$\cD 5.$] \hspace{2em} $\lim_{x\to\infty}g(x)/\psi(x)=0$.
\end{description}

\begin{lemma}\label{L2Dlemma}
If $G$ is bounded then the operator $G$ maps $L^2(0,\infty)$ one-one onto $\cD$.
\end{lemma}

\begin{proof}
If $f\in L^2(0,\infty)$ and $g=Gf$ then
\begin{eq}
g(x)=\psi(x)\int_x^\infty \phi(y)f(y)\, \rmd y%
+\phi(x)\int_0^x \psi(y)f(y)\, \rmd y\label{L2DeqA}
\end{eq}
for all $x\geq 0$. The boundedness of $G$ and (\ref{L2DeqA}) imply that $g\in L^2(0,\infty) \cap\cW$ and $g(0)=0$. Differentiating this formula yields
\begin{eq}
g^\pr(x)=\psi^\pr(x)\int_x^\infty \phi(y)f(y)\, \rmd y%
+\phi^\pr(x)\int_0^x \psi(y)f(y)\, \rmd y \label{L2DeqD}
\end{eq}
for almost all $x\geq 0$. By combining (\ref{L2DeqA}) and (\ref{L2DeqD}) one obtains
\begin{eq}
\phi(x) g^\pr(x)- \phi^\pr(x)g(x)=\int_x^\infty \phi(y)f(y)\, \rmd y\label{L2DeqB}
\end{eq}
for almost all $x\geq 0$. This formula implies that $\phi g^\pr- \phi^\pr g\in\cW$ with
\begin{eq}
( \phi g^\pr-\phi^\pr  g )^\pr=-\phi f.\label{L2DeqC}
\end{eq}
We conclude that $g\in\cD$ by combining these results with Lemma~\ref{g-asympt-lemma}.

The fact that $G$ is one-one was proved in Lemma~\ref{Gboundedlemma}. The proof that $G$ maps $L^2(0,\infty)$ onto $\cD$ is harder.
Given $h\in \cD$ we put $f=\phi^{-1}(\phi^\pr h -\phi h^\pr)^\pr$. Applying condition~$\cD 4$ to $h$ implies that $f\in L^2(0,\infty)$. If we now define $g=Gf$ and then $k=g-h$ we see that $k\in\cD$ and $\phi^{-1}( \phi^\pr k-\phi k^\pr)^\pr=0$; this uses (\ref{L2DeqC}). Therefore there exists a constant $c$ such that $\phi k^\pr-k \phi^\pr=c$. If one rewrites this in the form $(k/\phi)^\pr=c/\phi^2$ and then integrates one obtains
\[
k(x)/\phi(x)=c\int_0^x \phi(y)^{-2}\, \rmd y+d
\]
for all $x\geq 0$, and then $k=c\psi+d\phi$. Lemma~\ref{psi/phi} and condition~$\cD5$ (applied to $k$) together imply that $c=0$. By putting $x=0$ one then sees that $d=0$. Therefore $k=0$ and $h=g=Gf$. Therefore $G$ maps $L^2(0,\infty)$ onto $\cD$.
\end{proof}

\begin{corollary}
Suppose that $h\in C^2[0,\infty)$, that $h(0)=0$ and that $h(x)$ is constant for all large enough $x$. Then $\phi h\in \cD$ and
\begin{eq}
H(\phi h)=-\phi^{-1}(\phi^2 h^\pr)^\pr=-\phi h^{\pr\pr}-2\phi^\pr h^\pr.\label{Hweighted}
\end{eq}
\end{corollary}

\begin{note}
Forgetting domain questions, (\ref{Hweighted}) has the higher dimensional analogue
\[
H f= -\phi^{-1}\grad\cdot \left(\phi^{2} \grad( \phi^{-1}f)\right)
\]
where $f\in L(X,\rmd^n x)$ and $X$ is an open subset of $\R^n$. The operator $H$ is unitarily equivalent to $K$ acting in $L^2(X, \phi^2 \rmd x)$ according to the formula
\[
Kg= \phi^{-2}\grad\cdot \left( \phi^{2} \grad g\right).
\]
The paper \cite{wiel} obtains conditions for $K$ to be essentially self-adjoint, while \cite{EBD1} treats the same operator by using quadratic form techniques. See also \cite[Section~4.2]{HKST}.
\end{note}

\begin{theorem} \label{Htheorem}
The operator $H=G^{-1}$ is self-adjoint on $\cD=\Dom(H)$ and
is given by the formula
\begin{eq}
Hg=\phi^{-1}( \phi^\pr g- \phi g^\pr )^\pr\label{evalHg}
\end{eq}
for all $g\in\cD$. Moreover $\Spec(H)\subseteq [\norm G\norm^{-1},\infty)$. Every $g\in\cD$ satisfies the boundary condition $g(0)=0$. Moreover
\begin{eqnarray}
\la Hg_1,g_2\ra &=&\int_0^\infty %
(\phi g_1^\pr-\phi^\pr g_1  )
(\phi   \overline{g_2^\pr} - \phi^\pr \overline{g_2} )%
\phi ^{-2}\, \rmd x \nonumber\\
&=& \int_0^\infty %
\left(\frac{g_1}{\phi}\right)^\pr %
\left(\frac{\overline{g_2}}{\phi}\right)^\pr %
\phi^2\, \rmd x. \label{evalHg1g2}
\end{eqnarray}
for all $g_1,\, g_2\in\cD$.
\end{theorem}

\begin{proof}
The proof of the formula (\ref{evalHg}) and of the self-adjointness of $H$ on $\cD=\Dom(H)$ follows from the self-adjointness of $G$ and Lemma~\ref{L2Dlemma}; see (\ref{L2DeqC}). The statement about $\Spec(H)$ follows by applying the spectral theorem together with %
Lemma~\ref{Gboundedlemma}. The proof of (\ref{evalHg1g2}) uses integration by parts, but we need to verify that the functions involved have the required regularity properties.

\textbf{Step~1} Suppose that $0\leq f_r\in L^2(0,\infty)$ and $\supp(f_r)\subseteq [0,a]$ for $r=1,\, 2$. Putting $g_r=Gf_r$ for $r=1,\, 2$ and applying (\ref{L2DeqA}), we obtain $g_r\geq 0$ and $g_r(x)=c_r \phi(x)$ for all $x\geq a$. By combining the last identity with (\ref{L2DeqB}), we obtain $0\leq \phi g_r^\pr-\phi^\pr g_r\in\cW$ and
\[
\phi(x) g_r^\pr(x)-\phi^\pr(x) g_r(x)=0
\]
for all $x\geq a$.

We now put
\[
F=(\phi g_1^\pr-\phi^\pr g_1)\frac{g_2}{\phi}.
\]
The assumed properties of $g_1,\ g_2,\, \phi$ and our above results imply that $F\in \cW$, $F\geq 0$, $F(0)=0$ and $F(x)=0$ for all $x\geq a$. Moreover $F^\pr=h_1+h_2$ where
\begin{eqnarray*}
h_1&=&(\phi g_1^\pr-\phi^\pr g_1 )^\pr%
\frac{g_2}{\phi}\\
&=& -(Hg_1) g_2,\\
h_2&=& ( \phi g_1^\pr-\phi^\pr g_1)%
\left( \frac{g_2}{\phi}\right)^\pr\\
&=& (  \phi g_1^\pr-\phi^\pr g_1)%
 (  \phi g_2^\pr-\phi^\pr g_2)\phi^{-2}.
\end{eqnarray*}
The functions $h_1$ and $h_2$ are in $L^2_{{\rm loc}}[0,\infty)$ and both vanish for all $x\geq a$, so
\begin{eqnarray*}
\lefteqn{\int_0^\infty h_1(x)\, \rmd x+\int_0^\infty h_2(x)\, \rmd x}\\
&=& \int_0^{a}( h_1(x)+ h_2(x))\, \rmd x\\
&=& F(a)-F(0)\\
&=& 0.
\end{eqnarray*}
This proves that
\begin{eqnarray}
\la Hg_1,g_2\ra &=&\int_0^\infty %
(\phi g_1^\pr-\phi^\pr g_1  )
(\phi   g_2^\pr - \phi^\pr g_2 )%
\phi ^{-2}\, \rmd x,\label{quadg1g2}
\end{eqnarray}
the integrand being non-negative.

\textbf{Step~2} Let $0\leq f_r\in L^2(0,\infty)$ and put $g_r=Gf_r$ for $r=1,\, 2$. For each $n\in \N$ put $f_{r,n}=f_r\chi_{[0,n]}$ and $g_{r,n}=Gf_{r,n}$. Then Case~1 yields
\begin{eq}
\la Hg_{1,n},g_{2,n}\ra =\int_0^\infty %
(\phi g_{1,n}^\pr-\phi^\pr g_{1,n}  )
(\phi   g_{2,n}^\pr - \phi^\pr g_{2,n} )%
\phi ^{-2}\, \rmd x \label{quadg1g2n}
\end{eq}
for all $n\in \N$. Moreover
\[
\lim_{n\to\infty}\la Hg_{1,n},g_{2,n}\ra %
=\lim_{n\to\infty}\la f_{1,n},Gf_{2,n}\ra %
=\la f_{1},Gf_{2}\ra =\la Hg_{1},g_{2}\ra .
\]
An application of (\ref{L2DeqB}) implies that the integrand in  (\ref{quadg1g2n}) is non-negative and monotone increasing in $n$, so the monotone convergence theorem yields (\ref{quadg1g2}) for the present class of $g_1,\, g_2$.

\textbf{Step~3} If $f_r\in L^2(0,\infty)$ then one may write $f_r=\sum_{s=0}^3 i^s f_{r,s}$ where $0\leq f_{r,s}\in L^2(0,\infty)$ for all $r,\, s$. Putting $g_{r,s}=Gf_{r,s}$ and $g_r=Gf_r$, the bilinearity of both sides of (\ref{evalHg1g2}) now implies the general validity of this equation.
\end{proof}

\begin{note}
The proof that $0\notin\Spec(H)$ depends on the assumption that $\phi\in L^2(0,\infty)$. If, for example, $Hf=-f^{\pr\pr}+Vf$ where the potential $V$ has compact support, then the relevant fundamental solution $\phi$ of $H\phi =0$ would equal a positive constant for all large enough $x$, and one would have to add a sufficiently large positive constant to $V$ to obtain a fundamental solution $\phi\in L^2(0,\infty)$.
\end{note}

\section{Robin boundary conditions}\label{sectRBC}

The analysis above leads to a second order differential operator $H$ that satisfies Dirichlet boundary conditions at $0$. Other boundary conditions can be accommodated if one replaces $\psi$ by $\xi$, where
\begin{eq}
\xi(x)=\psi(x)+\gam \phi(x)\label{xidef}
\end{eq}
for all $x\geq 0$. Since we allow $\gam$ to take any non-zero complex value below, the operator $H_\gam$ constructed below need not be self-adjoint. Nevertheless, $\xi^\pr(x)\phi(x)-\phi^\pr(x)\xi(x)=1$ for almost all $x\geq 0$, and this proves enough to adapt our previous analysis. We define the new Green function by
\[
G_\gam(x,y)=\begin{choices}
\xi(x)\phi(y)&\mbox{if } 0\leq x\leq y <\infty,\\
\phi(x)\xi(y)&\mbox{if } 0\leq y\leq x <\infty.
\end{choices}
\]

\begin{theorem}\label{Ggamtheorem}
If $\gam\in\C\backslash \{ 0\}$, then under the assumptions of Lemma~\ref{Gboundedlemma}, $G_\gam$ is the kernel of a bounded operator whose kernel is $0$ and whose range is dense. The spectrum of the closed densely defined operator $H_\gam=G_\gam^{-1}$ does not contain $0$. The essential spectrum of $H_\gam$ does not depend on $\gam$. If $\gam\in \R\backslash \{ 0\}$ then $H_\gam$ is self-adjoint.
\end{theorem}

\begin{proof}
The definition leads directly to the formula
\[
G_\gam(x,y)=G(x,y)+\gam \phi(x)\phi(y),
\]
so $G_\gam$ is obtained from $G$ by adding a bounded rank one perturbation. The boundedness of the operator $G_\gam$ follows immediately. The proof that $\Ker(G_\gam)=0$ is the same as that for $\Ker(G)=0$ in Lemma~\ref{Gboundedlemma}, provided $\psi$ is replaced by $\xi$. The density of the range of $G_\gam$ is equivalent to the vanishing of the kernel of $(G_\gam)^\ast=G_{\overline{\gam}}$. The boundedness of $G_\gam$ directly implies that $H_\gam$ is closed. The fact that $G-G_\gam$ has rank $1$ implies that $H$ and $H_\gam$ have the same essential spectrum by \cite[Cor.~11.2.3]{LOTS}. If $\gam$ is real then $G_\gam$ is self-adjoint, so $H_\gam$ is also self-adjoint by the spectral theorem.
\end{proof}

The simplest description of the domain of $H_\gam$ involves a condition of the form $g^\pr(0)=\sig g(0)$, but this does not make sense unless $g^\pr(0)$ can be evaluated. This problem can be resolved by formulating the boundary condition in a more complicated manner. We make some further comments about this in Note~\ref{phiprcont}.

We say that $g\in\cD_\gam$ if $g$ satisfies the conditions $\cD 1$, $\cD 3$, $\cD 4$, $\cD 5$ and

$\cD 2\gam$.\hspace{2em} We assume that
\[
(\phi g^\pr-\phi^\pr g)(0)=\frac{g(0)}{\gam \phi(0)}.
\]
This makes sense because $\phi g^\pr-\phi^\pr g\in C[0,\infty)$ by condition~$\cD 3$.

Note that we do not replace $\psi$ by $\xi$ in condition~$\cD 5$.

\begin{theorem} \label{Hgamtheorem}
One has $\Dom(H_\gam)=\Ran(G_\gam)=\cD_\gam$ for all $\gam\in\C\backslash \{ 0\}$. If $g\in\cD_\gam$ then
\begin{eq}
H_\gam g= \phi^{-1}(\phi^\pr g- \phi g^\pr)^\pr.\label{Hgamg}
\end{eq}
If $g_r\in \cD_\gam$ for $r=1,\, 2$ then
\begin{eqnarray}
\la H_\gam g_1,g_2\ra &=&\int_0^\infty %
(\phi g_1^\pr-\phi^\pr g_1  )
(\phi   \overline{g_2^\pr} - \phi^\pr \overline{g_2} )%
\phi ^{-2}\, \rmd x%
+\frac{g_1(0)\overline{g_2(0)}}{\gam \phi(0)^2}.%
\label{quadg1g2gam}
\end{eqnarray}

\end{theorem}

\begin{proof}
The proofs of ($\cD 1$), ($\cD 3$) and ($\cD 4$) are very similar to the proofs in Lemma~\ref{L2Dlemma}.

The analogues of (\ref{L2DeqA}) and (\ref{L2DeqD}) in the present context are
\begin{eq}
g(x)=\xi(x)\int_x^\infty \phi(y)f(y)\, \rmd y%
+\phi(x)\int_0^x \xi(y)f(y)\, \rmd y\label{L2DeqAgam}
\end{eq}
and
\begin{eq}
g^\pr(x)=\xi^\pr(x)\int_x^\infty \phi(y)f(y)\, \rmd y%
+\phi^\pr(x)\int_0^x \xi(y)f(y)\, \rmd y. \label{L2DeqDgam}
\end{eq}
These imply that
\begin{eq}
(\phi g^\pr-\phi^\pr g)(x)=\int_x^\infty\phi(y)f(y)\, \rmd y.
\label{L2Dgphi}
\end{eq}
By evaluating (\ref{L2DeqAgam}) and (\ref{L2Dgphi}) at $x=0$ we obtain
\begin{eqnarray*}
&&g(0)= c \,\xi(0)\,\hspace{0.25em}\!=\!\,c\,\gam \phi(0),\\
&&(\phi g^\pr-\phi^\pr g)(0)=c,
\end{eqnarray*}
where
\[
c=\int_0^\infty \phi(y) f(y)\, \rmd y.
\]
These equations imply condition~$\cD2\gam$. The proof of ($\cD 5$) uses
\[
(Gf)(x)=(G_\gam f)(x)+\gam \phi(x)\int_0^\infty \phi(y)f(y)\, \rmd y
\]
together with Lemmas~\ref{psi/phi} and \ref{g-asympt-lemma}.

This completes the proof that $\Ran(G_\gam)\subseteq \cD_\gam$. The proof of equality follows that of Lemma~\ref{L2Dlemma}. Given $h\in \cD_\gam$ one defines $k=h-g\in\cD_\gam$ where $g\in \Ran(G_\gam)$ as before, and deduces that $k=c\psi+d\phi$. Lemma~\ref{psi/phi} and condition~$\cD 5$ (applied to $k$) together imply that $c=0$. Applying condition~$\cD 2\gam$ to $k=d\phi$ now yields
\[
0=(\phi k^\pr-\phi^\pr k)(0)=\frac{k(0)}{\gam \phi(0)}=\frac{d}{\gam}.
\]
Therefore $d=0$ and $k=0$. This proves that $h=g\in \Ran(G_\gam)$, and hence that $\Ran(G_\gam)=\cD_\gam$. The formula (\ref{Hgamg}) follows directly.

We prove (\ref{quadg1g2gam}) by using the similar formula in Theorem~\ref{Htheorem}. Given $g_r\in \cD_\gam$ there exist $f_r\in L^2(0,\infty)$ such that $G_\gam f_r=g_r$. We put $h_r=Gf_r\in \cD$, so that $g_r=h_r+\gam \la f_r,\phi\ra \phi$. This implies that $\phi g_r^\pr-\phi^\pr g_r=\phi h_r^\pr-\phi^\pr h_r$. We have
\begin{eqnarray*}
\la H_\gam g_1,g_2\ra&=& \la f_1, G_\gam f_2\ra\\
&=& \la  f_1,Gf_2\ra +\overline{\gam } \la f_1,\phi\ra\, \la \phi,f_2\ra\\
&=& c_1+c_2
\end{eqnarray*}
where
\begin{eqnarray*}
c_1&=& \la  f_1,Gf_2\ra  \\
&=& \la H h_1,h_2\ra\\
&=& \int_0^\infty %
(\phi h_1^\pr-\phi^\pr h_1  )
(\phi   \overline{h_2^\pr} - \phi^\pr \overline{h_2} )%
\phi ^{-2}\, \rmd x\\
&=& \int_0^\infty %
(\phi g_1^\pr-\phi^\pr g_1  )
(\phi   \overline{g_2^\pr} - \phi^\pr \overline{g_2} )%
\phi ^{-2}\, \rmd x,
\end{eqnarray*}
and
\[
c_2 =\overline{\gam } \la f_1,\phi\ra\, \la \phi,f_2\ra.
\]

We have already proved that
\[
\la f_r,\phi\ra=\int_0^\infty f_r(y)\phi(y)\, \rmd y = \frac{g_r(0)}{\gam \phi(0)}
\]
for $r=1,\, 2$. This implies that
\[
c_2=\overline{\gam}\,\frac{g_1(0)}{\gam \phi(0)}%
\,\frac{\overline{g_2(0)}}{\overline{\gam} \phi(0)}%
=\frac{g_1(0)\overline{g_2(0)}}{\gam\phi(0)^2}.
\]
The proof of (\ref{quadg1g2gam}) is completed by adding this to the formula for $c_1$.
\end{proof}

\begin{note}\label{phiprcont}
If there exists $\eps >0$ such that $\phi^\pr$ is continuous on $[0,\eps]$, then the same applies to $\psi^\pr$ by (\ref{psidef}) and to $\xi^\pr$ by (\ref{xidef}). By using (\ref{L2DeqAgam}) and (\ref{L2DeqDgam}) we deduce that $g$ and $g^\pr$ are continuous on $[0,\eps]$. Therefore condition~$\cD 2\gam$ may be rewritten in the form $g^\pr(0)=\sig g(0)$ where
\begin{eq}
\sig=\frac{\phi^\pr(0)}{\phi(0)}+\frac{1}{\gam \phi(0)^2}.  \label{gam2sig}
\end{eq}
\end{note}

\section{Spectral asymptotics}\label{sectSA}

In this section we assume that $H$ satisfies Dirichlet boundary conditions at $0$ and that it is defined as described in Section~\ref{sectDBC}. A detailed analysis of the spectral asymptotics has been given in \cite{HM,KM,SS2} for the analogous problem in $L^2(0,1)$ subject to various boundary conditions, where the operator $H$ is not required to be self-adjoint. We also mention the very thorough analysis of conditions for the existence and completeness of the wave operators for rapidly oscillating potentials on $(0,\infty)$ given in \cite{white}.

Our analysis of the spectral asymptotics depends entirely on the behaviour of $\phi$ as $x\to\infty$. Our main insight is that rapid oscillations in $\phi$ have no effect on the existence of any essential spectrum or on the rate of growth of the eigenvalues $\lam_n$ of $H$ as $n\to\infty$, if $H$ has no essential spectrum. Spectral questions of this type can be reformulated in terms of the Green function $G$ and then proved by using the formula $G=LL^\ast$ proved in Lemma~\ref{Gboundedlemma}. Recall that if $\norm G\norm <\infty$ then
\[
(Lf)(x)=\int_0^x \frac{\phi(x)}{\phi(y)} f(y)\, \rmd y
\]
and $\norm L\norm=\norm G\norm^{1/2}$. We now suppose that $\phi_1,\, \phi_2$ are two functions satisfying the basic hypothesis of this paper and that there exists a constant $c>0$ such that
\begin{eq}
c^{-1}\leq \frac{\phi_2(x)}{\phi_1(x)}\leq c\label{phi1phi2}
\end{eq}
for all $x\geq 0$. We add subscripts to all of the entities associated with $\phi_1,\, \phi_2$ as necessary.

The conclusion of the following theorem may be improved under stronger hypotheses. For example, if (\ref{phi1phi2}) is replaced by
\[
\lim_{x\to\infty} \frac{\phi_2(x)}{\phi_1(x)}=1
\]
then one may prove that
\[
\lim_{x\to\infty} \frac{\mu_{2,n}}{\mu_{1,n}}=1.
\]

\begin{theorem}\label{Gspecasymp}
If (\ref{phi1phi2}) holds, then $G_1$ is a bounded operator on $L^2(0,\infty)$ if and only if  $G_2$ is bounded. Moreover $G_1$ is compact if and only if $G_2$ is compact. If this holds and $\mu_{r,n}$, $n=1,2,\ldots$, are the eigenvalues of $G_r$ written in decreasing order and repeated according to multiplicity, then
\[
c^{-4}\mu_{1,n}\leq \mu_{2,n}\leq c^4\mu_{1,n}
\]
for all $n$.
\end{theorem}

\begin{proof}
We first note that $G_r$ is bounded (resp. compact) if and only if $L_r$ is bounded (resp. compact). By examining the integral kernels one sees that $L_2=D^{-1}L_1D$ where $D$ is the bounded invertible operator given by
\[
Df=\frac{\phi_1}{\phi_2}f
\]
for all $f\in L^2(0,\infty)$. Therefore $L_2$ is bounded (resp. compact) if and only if $L_1$ is bounded (resp. compact).

Assuming that $G_r$ are compact we compare their eigenvalue asymptotics by using the formulae
\begin{eq}
G_1=LL^\ast, \hspace{3em} G_2=%
D^{-1}L_1D^2L_1^\ast D^{-1}.\label{G2formula}
\end{eq}
The assumptions relating $\phi_1$ and $\phi_2$ yield $c^{-2}G_2\leq G_3\leq c^2G_2$ in the sense of quadratic forms, where $G_3= D^{-1}L_1L_1^\ast D^{-1}$. Variational estimates then imply that $G_2$ and $G_3$ have the same eigenvalue asymptotics up to a factor of $c^{\pm 2}$. A very general lemma (see \cite[Problems~1.2.5, 1.2.6]{LOTS}) implies that $G_3$ has the same non-zero eigenvalues with the same multiplicities as $G_4= L_1^\ast D^{-2}L_1$. The same argument as before shows that $G_4$ has the same eigenvalue asymptotics up to a factor of $c^{\pm 2}$ as $G_5=L^\ast L$. This finally has the same eigenvalues  with the same multiplicities as $G_1$.
\end{proof}

Following the notation of Theorem~\ref{Gexpbound}, we put $\phi_1=\phi$ and $\phi_2=\rme^{-\sig}$ and obtain the following.

\begin{corollary}\label{HEScoro}
Let $\phi$ satisfy the hypothesis of Theorem~\ref{Gexpbound} and suppose that
\[
\lim_{x\to\infty}\left\{ (\sig^\pr(x))^2-\sig^{\pr\pr}(x)\right\}=+\infty.
\]
Then the essential spectrum of the operator $H$ considered in Section~\ref{sectDBC} is empty and the eigenvalues $\lam_n$ of $H$ grow at the same rate as those of the \Schrodinger operator $\tH$, where
\[
\tH f=-f^{\pr\pr}+\tV f
\]
and
\[
\tV= (\sig^\pr)^2-\sig^{\pr\pr}.
\]
\end{corollary}

\section{Scattering theory}\label{sectST}

A number of earlier papers have developed scattering theory in the context of rapidly oscillating potentials; see \cite{comb,matv,pear,skrig,white} and \cite[Apps.~2,~3 to XI.8]{RS3}. In this section we obtain analogous results starting from $\phi$; as usual we do not impose any local or global bounds on $\phi^\pr$.

Throughout the section we assume that
\begin{eq}
\phi(x)=\rme^{-cx-\zet(x)},%
\mbox{ where $c>0$ and $\zet\in \cW\cap L^\infty(0,\infty)$. }\label{phizeta}
\end{eq}
An application of Theorem~\ref{Gexpbound} with $\sig(x)=cx$ implies that $G(\cdot,\cdot)$ is the kernel of a bounded operator on $L^2(0,\infty)$. Our next theorem compares the operator $H$ constructed from $\phi$ as described in Section~\ref{sectDBC} with the operator $H_0f=-f^{\pr\pr}+c^2 f$ constructed by the same method from $\phi_0=\rme^{-cx}$.

\begin{theorem}\label{scatt-theorem}
If $\phi$ is defined by (\ref{phizeta}) and $\zet$ satisfies the further condition $|\zet(x)|\leq \nu(x)$ for all $x\geq 0$ where $\nu$ is a monotonic decreasing function lying in $L^1(0,\infty)$,
then the wave operators between $H$ and $H_0$ exist and are complete.
\end{theorem}

\begin{proof}\\
\textbf{Step~1} It is sufficient by the Kuroda-Birman theorem to prove that $G-G_0$ is a trace class operator; see \cite[Theorem~XI.9]{RS3}. If one writes (\ref{Mdef}) in the form
\[
(Mf)(x)=\langle f,\xi_x\ra
\]
where
\[
\xi_x(u)=\begin{choices}
\phi(u)/\phi(x) &\mbox{ if $u\geq x$,}\\
0&\mbox{ otherwise,}
\end{choices}
\]
then (\ref{G=MM}) takes the form
\[
G=\int_0^\infty | \xi_x\ra\la \xi_x|\, \rmd x .
\]

Therefore
\begin{eqnarray*}
\norm G-G_0\norm_{{\rm tr}}&\leq&\int_0^\infty %
\norm \, | \xi_x\ra\la \xi_x|  - | \xi_{0,x}\ra\la \xi_{0,x}|%
\, \norm_{{\rm tr}}\, \rmd x\\
&\leq& \int_0^\infty %
(\norm \xi_x\norm + \norm \xi_{0,x}\norm )%
\norm \xi_x- \xi_{0,x}\norm\, \rmd x.
\end{eqnarray*}
We need to estimate the terms in this bound.

\textbf{Step~2} It follows directly from their definitions that \[
\norm \xi_{0,x}\norm =\frac{1}{\sqrt{2c}}
\]
for all $x\geq 0$ and that
\begin{eqnarray*}
\norm \xi_x\norm^2 &=& \int_x^\infty \frac{\phi(u)^2}{\phi(x)^2} \, \rmd u\\
&=& \int_x^\infty \exp(-2c(u-x)-2\zet(u)+2\zet(x))\, \rmd u\\
&\leq  &  \int_x^\infty \exp(-2c(u-x)+4\norm\zet\norm_\infty )\, \rmd u\\
&=& \frac{\rme^{4\norm\zet\norm_\infty}}{2c}.
\end{eqnarray*}
Therefore
\[
\norm \xi_x\norm\leq \frac{\rme^{2\norm\zet\norm_\infty}}{\sqrt{2c}}
\]
for all $x\geq 0$. If $u\geq x\geq 0$ then
\begin{eqnarray*}
|\xi_x(u)-\xi_{0,x}(u)|&=&  \rme^{-c(u-x)}%
\left| \rme^{-\zet(u)+\zet(x)}-1\right|\\
&\leq & \rme^{-c(u-x)}\rme^{2\norm \zet\norm_\infty}%
\left| \zet(u)-\zet(x)\right| .
\end{eqnarray*}
This uses the elementary bound
\begin{eq}
|\rme^{\zet(x)-\zet(u)}-1|\leq \rme^{2\norm \zet\norm_\infty}|\zet(x)-\zet(u)|,\label{expzeta}
\end{eq}
valid for all $x,u\geq 0$. Therefore
\begin{eqnarray}
\norm \xi_x-\xi_{0,x}\norm^2&\leq&\rme^{4\norm \zet\norm_\infty}\int_x^\infty %
\rme^{-2c(u-x)}
\left| \zet(u)-\zet(x)\right|^2\, \rmd u\label{diffbound1}\\
&\leq&\rme^{4\norm \zet\norm_\infty}\int_x^\infty %
\rme^{-2c(u-x)}4\nu(x)^2\, \rmd u\nonumber\\
&=& \rme^{4\norm \zet\norm_\infty}%
\frac{4\nu(x)^2}{2c}.\label{diffbound2}
\end{eqnarray}

\textbf{Step~3} Putting the various bounds together, there exist constants $k_r$ such that
\[
\norm G-G_0\norm_{{\rm tr}}\leq k_1\int_0^\infty \norm \xi_x-\xi_{0,x}\norm\, \rmd x\leq k_2\int_0^\infty \nu(x)\, \rmd x <\infty.
\]
\end{proof}

\begin{corollary}\label{scatt-cor}
If $\phi$ is defined by (\ref{phizeta}) and $\zet$ satisfies the further condition $|\zet(x)|\leq k(1+x)^{-\alp}$ for some $k>0$ and all $x\geq 0$, then the wave operators between $H$ and $H_0$ exist and are complete provided $\alp>1$.
\end{corollary}

\begin{example}\label{scatt-ex}
If $\alp>0$ and $\zet(x)=(1+x)^{-\alp}$ for all $x\geq 0$ then
\[
|\zet(u)-\zet(x)|\leq (u-x)\alp(1+x)^{-\alp-1}
\]
provided $0\leq x\leq u$. Using (\ref{diffbound1}) we obtain
\[
\norm \xi_x-\xi_{0,x}\norm^2\leq k_1\int_x^\infty%
\rme^{-2c(u-x)}(u-x)^2\alp^2 (1+x)^{-2\alp-2}\, \rmd u\leq k_2 (1+x)^{-2\alp-2}.
\]
Therefore
\[
\int_0^\infty \norm \xi_x-\xi_{0,x}\norm\, \rmd x\leq k_3 \int_0^\infty (1+x)^{-\alp-1}\, \rmd x <\infty.
\]
We conclude that the wave operators exist and are complete for all $\alp >0$. The potential in this example is given by
\begin{eqnarray*}
V(x)&=&\frac{\phi^{\pr\pr}(x)}{\phi(x)}\\
&=& -\zet^{\pr\pr}(x)+(c+\zet^\pr(x))^2\\
&=&c^2-2c\alp(1+x)^{-\alp-1}+O((1+x)^{-\alp-2}).
\end{eqnarray*}
Therefore $V-c^2\in L^1(0,\infty)$ for all $\alp >0$. This is the traditional condition for existence and completeness of the wave operators.
\end{example}

\begin{note}
Example~\ref{scatt-ex} seems to suggest that the condition $\alp>1$ in Corollary~\ref{scatt-cor} is not optimal, but we conjecture that the corollary cannot be improved, and that the need for a stronger hypothesis is the price paid for not assuming any bounds on $\phi^\pr$, or equivalently on $\zet^\pr$. The benefit is that the result applies to a wide range of potentials, and in particular to many potentials that are rapidly oscillating as $x\to\infty$.

One may also obtain intermediate results if one assumes some global Holder continuity condition on $\zet$.
\end{note}

\textbf{Acknowledgements} I should like to thank Fritz Gesztesy, Sasha Pushnitski, Alex Sobolev and Gerald Teschl for very helpful comments.

E. B. Davies,\\
Department of Mathematics,\\
King's College London,\\
Strand,\\
London,WC2R 2LS,\\
UK

\vspace{6ex}

\begin{thebibliography}{99}

\bibitem{atk} Atkinson F. V., The asymptotic solution of second-order differential equations, Ann. Mat. Pura Appl. 37 (1954) 347-378.

\bibitem{AEZ} Atkinson, F. V., Everitt, W. N., Zettl, A., Regularization of a Sturm-Liouville problem with an interior singularity using quasi-derivatives, Diff. Int. Eqns. 1 (2) (1988) 213-221.

\bibitem{B-AD} Ben-Artzi, M. and Devinatz, A., Spectral and scattering theory for the adiabatic oscillator and related potentials, J. Math. Phys. 20 (4) (1979) 594-607.

\bibitem{BL} Benguria, R. and Loss, M., a simple proof of a theorem of Laptev and Leidl, Math. Research Lett. 7 195-203 (2000)

\bibitem{BE} Bennewitz, C. and Everitt, W. N., On second-order left-definite boundary value problems, pp. 31-67 in ``Ordinary Differential Equations and Operators (Dundee, 1982)'', Lecture Notes in Math. 1032, Springer, Berlin, 1983.

\bibitem{comb} Combescure, M., Spectral and scattering theory
for a class of strongly oscillating potentials, Comm. Math. Phys. 73 (1980) 43-62.

\bibitem{HKST} Davies, E. B., Heat Kernels and Spectral Theory, Camb. Univ. Press, 1989.

\bibitem{EBD1} Davies, E. B., Heat kernel bounds, conservation of probability and the Feller property, J. d'Analyse Math. 58 (1992) 99-119.

\bibitem{STDO} Davies, E. B., Spectral Theory and Differential Operators, Camb. Univ. Press, 1995.

\bibitem{LOTS} Davies, E. B., Linear Operators and their Spectra, Camb. Univ. Press, 2007.

\bibitem{DH} Davies, E. B. and Harrell II, E. M., Conformally flat Riemannian metrics, \Schrodinger operators and semiclassical approximation, J. Diff. Eqns. 66 (1987) 165-188.

\bibitem{EGNT} Eckhardt, J., Gesztesy, F., Nichols, R., Teschl, G., Weyl-Titchmarsh theory for Sturm-Liouville operators with distributional coefficients, 2011, in preparation.

\bibitem{ET} Eckhart, J. and Teschl, G.,
Sturm-Liouville operators with measure-valued coefficients, preprint, 2011.

\bibitem{HM} Hryniv, R. O. and Mykytyuk, Y. V., Eigenvalue asymptotics for Sturm-Liouville operators with singular potentials, J. Funct. Anal. 238 (2006) 27-57.

\bibitem{KM} Kappeler, T., M\"ohr, C., Estimates for periodic and Dirichlet eigenvalues of the \Schrodinger operator with singular potentials, J. Funct. Anal. 186 (2001) 62-91.

\bibitem{matv} Matveev, V. B. and Skriganov, M. M., Wave operators for the \Schrodinger equation with rapidly oscillating potential, Soviet Math. Dokl. Vol. 13 (1972), 185-188.

\bibitem{ming} Mingarelli, A. B., Volterra-Stieltjes integral equations and generalized ordinary differential expressions, Lecture Notes in Mathematics 989, Springer, Berlin, 1983.

\bibitem{nest}  Nesterov, P. N., Construction of the asymptotic behavior of solutions of the one-dimensional \Schrodinger equation with a rapidly oscillating potential. (Russian) Mat. Zametki 80 (2006), no. 2, 240-250; translation in Math. Notes 80 (2006), no. 1-2, 233-243.

\bibitem{pear} Pearson, D. B., Scattering theory for a class of oscillating potentials. Helv. Phys. Acta 52 (1979) 541-554.

\bibitem{RS3} Reed, M. and Simon, B., Methods of Modern Mathematical Physics, Vol.~3, Scattering Theory, Academic Press, New York, 1979.

\bibitem{RS4} Reed, M. and Simon, B., Methods of Modern Mathematical Physics, Vol.~4, Analysis of Operators, Academic Press, New York, 1978.

\bibitem{reid} Reid, W. T., Riccati Differential Equations, Academic Press, New York, 1972.

\bibitem{SS1} Savchuk, A. M. and Shkalikov, A. A., Sturm-Liouville operators with singular potentials, Math. Notes 66 (1999) 741-753.

\bibitem{SS2} Savchuk, A. M. and Shkalikov, A. A., On the eigenvalues of the Sturm-Liouville operator with potentials from Sobolev spaces. (Russian) Mat. Zametki 80 (2006) 864-884; translation in Math. Notes 80 (2006) 814-832.

\bibitem{skrig} Skriganov, M. M., The spectrum of a \Schrodinger operator with rapidly oscillating potential, (Russian), Boundary value problems of mathematical physics, Part~8, Trudy Mat. Inst. Steklov 125 (1973) 187-195. English version, Proc. Steklov Inst. Math. 125 (1973) 177-185.

\bibitem{tes} Teschl, G., Mathematical Methods in Quantum Mechanics, with Applications to \Schrodinger Operators, Graduate Studies in Mathematics, vol.~99, Amer. Math. Soc. 2009.

\bibitem{white} White, D. A. W., \Schrodinger operators with rapidly oscillating central potentials, Trans. Amer. Math. Soc. 275 (1983) 641-677.

\bibitem{wiel} Wielens, N., The essential self-adjointness of generalized \Schrodinger operators, J. Funct. Anal. 61 (1985) 98-115.








\end{thebibliography}
\end{document}